\documentclass[11pt]{amsart}
\usepackage{graphicx}
\usepackage{verbatim}
\usepackage{textcomp}
\usepackage{amssymb}
\usepackage{cite}
\usepackage{amsmath}
\usepackage{latexsym}
\usepackage{amscd}
\usepackage{amsthm}
\usepackage{mathrsfs}
\usepackage{xypic}
\usepackage{bm}
\usepackage{hyperref}
\usepackage{url}

\newtheorem{theorem}{Theorem}[section]
\newtheorem{lemma}{Lemma}[section]
\newtheorem{definition}{Definition}[section]
\newtheorem{remark}{Remark}[section]
\newtheorem{proposition}{Proposition}[section]

\numberwithin{equation}{section}

\allowdisplaybreaks[4]

\newcommand{\tbf}{\textbf}

\begin{document}
\title{Rigidity theorem for a static triple with half harmonic Weyl curvature}

\author{Li Chen }
\address{School of Mathematics and Statistics, Hubei University, 430062, Wuhan, P. R. China}
\email{chenli@hubu.edu.cn}
\author{Xi Guo }
\address{School of Mathematics and Statistics, Hubei University, 430062, Wuhan, P. R. China}
\email{guoxi@hubu.edu.cn}

\date{}

\maketitle

\begin{abstract}
\noindent  In this paper, we prove that the static triple $(M^4,g,V)$ with $\delta W^\pm=0$ and positive scalar curvature
must be the standard hemisphere.
\end{abstract}


\section{Introduction}
Let $(M^n,g)$ be a n-dimension complete Riemannian manifold. A static potential
is a non-trivial solution $V\in C^\infty(M)$ to the equation
\begin{equation}\label{7-4-1}
Hess V-\triangle Vg-VRic=0.
\end{equation}
In fact, the equation \eqref{7-4-1} was originally derived from the linearization of the
scalar curvature equation. Suppose $M$ is a compact manifold, we denote $\mathcal{G}$ be
the space of smooth Riemannian metrics on $M$.
It is known
that the scalar curvature R can be consider as a mapping $R: \mathcal{G}\rightarrow C^\infty(M)$.
Then the derivative  $dR$ at $g$  is given by
\begin{equation}
dR_g(h)=\frac{d}{dt}\Big|_{t=0}R_{g+th}=div_g^2h-\triangle tr_gh-\langle Ric,h\rangle_g.
\end{equation}
Using Stokes' theorem, the $L_2$-adjoint operator $dR_g^*$ of $dR_g$ with respect to the
canonical $L_2$-inner product defined by $g$ is
\begin{equation}
dR_g^*(V)=Hess_g V-\triangle_gV-VRic.
\end{equation}
So if $(M,g)$ admit a static potential, then $g$ is a singular point of the mapping $R$.

\begin{definition}[\cite{Am}]
A static triple is a triple $(M,g,V)$ consisting of a connected
n-dimensional smooth manifold $M$ with boundary $\partial M$ (possibly empty), a
complete Riemannian metric $g$ on $M$ and a static potential $V\in C^\infty(M)$ that
is non-negative and vanishes precisely on $\partial M$. Two static triples $(M_i, g_i, V_i)$,
$i = 1, 2$, are said to be equivalent when there exists a diffeomorphism $\phi:M_1\rightarrow M_2$
such that $\phi^*g_2=cg_1$ for some constant $c>0$ and $V_2\circ \phi=\lambda V_1$
for some constant $\lambda>0$.
\end{definition}
In \cite{K},  Kobayashi gave all examples of  connected complete conformally flat Riemannian
manifold which admits a non-trivial solution to \eqref{7-4-1}.
We state Theorem 3.1 of \cite{K} in the case of positive scalar curvature in dimension 4.
\begin{theorem}[\cite{K}]\label{thmK}
Let $(M^4, g, V)$ be a static
triple with  scalar curvature 12. If  $(M,g)$ is locally conformally flat,
then $(M^4,g,V)$ is covered by a static triple that is equivalent to one of the following:\\
i) The standard hemisphere $(\mathbb{S}^4_+, g_0, V=x_4)$.\\
ii) The standard cylinder over $S^3$ with the product metric,
$$\Big([0,\frac{\pi}{2}]\times\mathbb{S}^3,g=dt^2+\frac{1}{2}g_0,V=\frac{1}{2}sin2t\Big).$$
iii) For some $m\in(0,\frac{1}{\sqrt{3}})$,
$$\Big([r_h(m),r_c(m)]\times\mathbb{S}^3,g=\frac{dr^2}{1-r^2-\frac{2m}{r^2}}+r^2g_0,V=\sqrt{1-r^2-\frac{2m}{r^2}}\Big)$$
where $r_h(m)<r_c(m)$ are the positive zeroes of V.
Here a static triple $(\tilde{M},\tilde{g},\tilde{V})$ covers a static triple $(M,g,V)$ when
there exists a covering map $\pi : M\rightarrow M$ such that $\tilde{g}=\pi^*g$ and $\tilde{V}=\pi\circ V$.
\end{theorem}
Motivated by the work of Kobayashi and Obata \cite{K,K-O},  Miao and Tam studied the M-T metric in \cite{MT}, which satisfies the equation
\begin{equation}\label{MTmetric}
Hess_g V-\triangle_gV-VRic=g,
\end{equation}
under Einstein or locally conformally flat
assumptions. While inspired by the trend developed by Cao and Chen in \cite{CC},
Barros,  Diogenes and Ribeiro  weaken the assumption  \cite{BDR}.
They proved that for a Bach-flat simply connected, compact manifold with M-T metric, if its
boundary isometric to a standard sphere $\mathbb{S}^3$, then it should
isometric to a geodesic ball in a simply connected space form.

In fact, it is well-known that on four-dimensional manifold $M^4$ the bundle of
two-forms splits $\Lambda^2=\Lambda_+^2+\Lambda_-^2$ into the $+1$-eigenspace of the Hodge $*$-operator
(self-dual two-forms) and $-1$-eigenspace (anti-self-dual two-forms). Then the Weyl
curvature tensor $W=W^++W^-$, where $W^\pm :  \Lambda_\pm^2\rightarrow\Lambda_\pm^2$.
and we say $W^\pm$ is harmonic if $\delta W^\pm=0$, where $\delta$ is the divergence.
In \cite{WWW}, J. Wu, P. Wu and Wylie proved that a four-dimensional gradient shrinking Ricci soliton with $\delta W^\pm=0$
is either Einstein, or a finite quotient of $\mathbb{S}^3\times \mathbb{R}$, $\mathbb{S}^2\times \mathbb{R}^2$, or $\mathbb{R}^4$.
Inspired by their work, we consider the 4-dimensional static triple  with $\delta W^\pm=0$ , our main results are the following.
\begin{theorem}\label{mainthm}
Let $(M^4,g,V)$ be a compact simply connected static triple with positive scalar curvature  and $\delta W^\pm=0$.
If the boundary isometric to a standard sphere $\mathbb{S}^3$,
then  $(M^4,g,V)$   is equivalent to the standard hemisphere $(\mathbb{S}^4_+, g_{can}, V=x_4)$.
\end{theorem}

\begin{theorem}\label{thm2}
Let $(M^4,g,V)$ be a compact connected static triple with positive scalar curvature  and $\delta W^\pm=0$.
If M has 2 connected components and each connected component of $\partial M$ is isometric to a sphere,
then  $(M^4,g,V)$ is covered by ii) or iii) in Theorem \ref{thmK}.
\end{theorem}

\medskip\noindent
{\bf Acknowledgement.}
The second author thanks Prof. Xiaodong Wang for his helpful discussion.
 This research was also supported in part by
Hubei Key Laboratory of Applied Mathematics (Hubei University).

\section{Preliminaries}

\subsection{Setting and General facts}

\

For later convenience, we first state our conventions on Riemann
Curvature tensor and derivative notation. Let $M$ be a smooth
manifold and $g$ be a Riemannian metric on $M$ with Levi-Civita
connection $\nabla$. For a $(s, r)$-tensor field $\alpha$ on $M$,
its covariant derivative $D \alpha$ is a $(s, r+1)$-tensor field
given by
\begin{eqnarray*}
&&\nabla \alpha(Y^1, .., Y^s, X_1, ..., X_r, X)
\\&=&\nabla_{X} \alpha(Y^1, .., Y^s, X_1, ..., X_r)\\&=&X(\alpha(Y^1, .., Y^s, X_1, ..., X_r))-
\alpha(\nabla_X Y^1, .., Y^s, X_1, ..., X_r)\\&&-...-\alpha(Y^1, ..,
Y^s, X_1, ..., \nabla_X  X_r),
\end{eqnarray*}
the coordinate expression of which is denoted by
$$\nabla \alpha=(\alpha_{k_{1}\cdot\cdot\cdot
k_{r}, k_{r+1}}^{l_{1}\cdot\cdot\cdot l_{s}}).$$ We can continue to
define the second covariant derivative of $\alpha$ as follows:
\begin{eqnarray*}
&&\nabla^2 \alpha(Y^1, .., Y^s, X_1, ..., X_r, X, Y)
=(\nabla_{Y}(\nabla\alpha))(Y^1, .., Y^s, X_1, ..., X_r, X),
\end{eqnarray*}
the coordinate expression of which is denoted by
$$\nabla^2 \alpha=(\alpha_{k_{1}\cdot\cdot\cdot
k_{r}, k_{r+1}k_{r+2}}^{l_{1}\cdot\cdot\cdot l_{s}}).$$ In
particular, for a function $u: M\rightarrow \mathbb{R}$, we have the
following important identity
$$\nabla^2 f(X, Y)=YX(f)-(\nabla_{Y}X)f.$$
Similarly, we can also define the higher order covariant derivative
of $\alpha$:
$$\nabla^3 \alpha=\nabla(\nabla^2 \alpha), ... ,$$
and so on. Our convention for the Riemannian curvature (3,1)-tensor
Rm is defined by
\begin{equation*}
Rm(X, Y)Z=-\nabla_{X}\nabla_{Y}Z+\nabla_{Y}\nabla_{X}Z+\nabla_{[X,Y]}Z.
\end{equation*}
By picking a local coordinate chart $\{x^i\}_{i=1}^{n}$ of $M$, the
component of the (3,1)-tensor $Rm$ is defined by
\begin{equation*}
Rm\bigg({\frac{\partial}{\partial x^i}}, {\frac{\partial}{\partial
x^j}}\bigg){\frac{\partial}{\partial x^k}} = R_{ijk}^{\ \ \
l}{\frac{\partial}{\partial x^l}}
\end{equation*}
and $R_{ijkl}\doteq g_{lm}R_{ijk}^{\ \ \ m}$.  Ricci curvature and scalar curvature are given by $R_{ik}\doteq R_{ijk}^{\ \ \ j}$
and $R\doteq R_{ik}g^{ik}$ respectively. Then, we have the
standard commutation formulas (Ricci identities):
\begin{equation}\label{RI}
\aligned
\alpha_{k_{1}\cdot\cdot\cdot k_{r},\ j i}^{l_{1}\cdot\cdot\cdot
l_{s}}-\alpha_{k_{1}\cdot\cdot\cdot k_{r},\ i
j}^{l_{1}\cdot\cdot\cdot l_{s}}=&\sum_{a=1}^{r}R^{\ \ \ m}_{ijk_{a}}
\alpha_{k_{1}\cdot\cdot\cdot k_{a-1}m k_{a+1}\cdot\cdot\cdot
k_{r}}^{l_{1}\cdot\cdot\cdot l_{s}}\\
&-\sum_{b=1}^{s}R^{\ \ \
l_b}_{ijm} \alpha_{k_{1}\cdot\cdot\cdot k_{r}}^{l_{1}\cdot\cdot\cdot
l_{b-1}m l_{b+1}\cdot\cdot\cdot l_{r}},
\endaligned
\end{equation}
where $\alpha_{k_{1}\cdot\cdot\cdot k_{r},\ j
i}^{l_{1}\cdot\cdot\cdot l_{s}}$ denote the second covariant
derivative of the tensor $\alpha_{k_{1}\cdot\cdot\cdot
k_{r}}^{l_{1}\cdot\cdot\cdot l_{s}}$.
And the Weyl curvature tensor of $(M, g)$ is defined by
\begin{equation}\label{26}
W_{ijkl}=R_{ijkl}-(P_{ik}g_{jl}-P_{il}g_{jk}+P_{jl}g_{ik}-P_{jk}g_{il}),
\end{equation}
where $P$ is  $Schouten$ tensor given by
\begin{equation}\label{28}
P_{ij}=\frac{1}{n-2}(R_{ij}-\frac{R}{2(n-1)}g_{ij}).
\end{equation}

\subsection{a static triple}
\begin{lemma}\cite{Bou,FM}
The static equation \eqref{7-4-1} is equivalent to the equations
\begin{eqnarray}
\label{1}&&Hess V=V(Ric-\lambda g),\\
\label{2}&&\triangle V=-\lambda V.
\end{eqnarray}
where the scalar curvature $R=(n-1)\lambda$ is constant. Moreover, if  $(M, g)$
admits a static potential $V\in C^\infty(M)$, then\\
i) $0$ is a regular value of V;\\
ii) ${V = 0}$ is totally geodesic;\\
iii) $|\nabla V|$ is locally constant and positive on $\{V = 0\}$.
\end{lemma}
So without loss of generality, when $n=4$, we assume $R=12$ and $|\nabla V|=1$ on $\{V = 0\}$.
In this case, the static equation \eqref{7-4-1} equivalent to
\begin{eqnarray}
\label{3}&&Hess V=V(E-g),\\
\label{4}&&\triangle V=-4V.
\end{eqnarray}
where $E=Ric-\frac{R}{n}g$ is the Einstein tensor.
We introduce the covariant 3-tensor $T_{ijk}$ (see (2.12) in \cite{BDR}):
\begin{equation*}
\aligned
T_{ijk}=&\frac{n-1}{n-2}(V_kR_{ij}-V_jR_{ik})
-\frac{R}{n-2}(g_{ij}V_k-g_{ik}V_j)\\&+\frac{1}{n-2}(V^sR_{ks}g_{ij}-V^sR_{js}g_{ik}),
\endaligned
\end{equation*}
which can be written as
\begin{equation}
T_{ijk}=\frac{n-1}{n-2}(V_kE_{ij}-V_jE_{ik})+\frac{1}{n-2}(V^sE_{ks}g_{ij}-V^sE_{js}g_{ik}).
\end{equation}

\begin{remark}
The indices of the tensor $T$ above is different with the one in \cite{BDR}.
 It is important to highlight that $T_{ijk}$ was defined similarly to the tensor $D$ in \cite{CC}.
\end{remark}

For convenient, now we omit the summation and identification through the duality defined by metric,
e.g. $W_{sijk,s}$ stands for $\sum_sW^s_{ijk,s}$.

\begin{lemma}
When $n=4$, if $(M,g,V)$ is a static triple, then
\begin{equation}
\aligned
T_{ijk}=V_sW_{sijk}+2VW_{sijk,s}.
\endaligned
\end{equation}
\end{lemma}
The proof is easy, we omit it here.

\subsection{Half Weyl curvature}

\

For a 4-dimension Riemannian manifold, let $\{e_1, e_2, e_3, e_4\}$ be an orthonormal basis of $TM$, for any pair $(ij)$, $1\leq i<j\leq4$, denote $(i'j')$to be the dual of $(ij)$,
i.e., the pair such that $e_i\wedge e_j \pm e_{i'}\wedge e_{j'} \in \Lambda_\pm^2$. In other words, $(iji'j')=\sigma(1234)$
for some even permutation. It is well known that $W_{ijkl}=W_{i'j'k'l'}$, then
$$W^\pm_{ijkl}=\frac{1}{2}(W_{ijkl}\pm W_{ijk'l'}).$$
And we have the following
Weitzenbock formula (See (16.73) in \cite{Be}).
\begin{proposition}\label{Ww}\cite{Be}
Every orient Riemannian 4-manifold with $\delta W^\pm=0$ satisfies
\begin{equation}\label{w}
\triangle|W^\pm|^2=2|\nabla W^\pm|^2+R|W^\pm|^2-144\det W^\pm.
\end{equation}
\end{proposition}

\begin{remark}
The defintion of $|W^\pm|$ in \cite{Be} is $\frac{1}{4}\sum_{ijkl}(W^\pm_{ijkl})^2$, and in our paper
is $\sum_{ijkl}(W^\pm_{ijkl})^2$. Besides the laplacian differs from Besse¡¯s by a sign.
\end{remark}

And likely in \cite{WWW}, we also define
\begin{equation}\label{6-26-1}
T^\pm_{ijk}=V_sW^\pm_{sijk}+2VW^\pm_{sijk,s}.
\end{equation}
We get the following lemmas.
\begin{lemma}\label{6-25-8}
Let $(M,g,V)$ be a static triple, then
\begin{equation}\label{6-26-2}
\aligned
T_{ijk}^\pm=&\frac{3}{4}(V_kE_{ij}-V_jE_{ik})\pm\frac{3}{4}(V_{k'}E_{ij'}-V_{j'}E_{ik'})\\
&-\frac{1}{4}(V_sE_{sj}g_{ik}-V_sE_{sk}g_{ij})\mp\frac{1}{4}(V_sE_{sj'}g_{ik'}-V_sE_{sk'}g_{ij'}),
\endaligned
\end{equation}
in particular,
\begin{equation}\label{7-9-6}
|T|^2=2|T^+|^2=2|T^-|^2.
\end{equation}
\end{lemma}

\begin{lemma}\label{lem1}
If $\delta W^\pm=0$ and $\nabla V\neq0$ at $p\in M$, then $\nabla V$ is an eigenvector of $E$.
\end{lemma}
In fact, the proof of Lemma \ref{6-25-8} and Lemma \ref{lem1} is similar with Lemma 2.3 and Lemma 2.4 in \cite{WWW}, so we omit it here.

\begin{lemma}\label{lem6-19-1}
If $\nabla V\neq0$ at $p\in M$, let $e_1=\nabla V/|\nabla V|$, and choose $\{e_1,e_2,e_3,e_4\}$ be the unit eigenvectors  of $E$.
Then,
\begin{eqnarray}
\label{6-25-1}W^{\pm}_{1212}&=-\frac{1}{4}(E_{11}+3E_{22}),\\
\label{7-9-1}W^{\pm}_{1313}&=-\frac{1}{4}(E_{11}+3E_{33}),\\
\label{7-9-2}W^{\pm}_{1414}&=-\frac{1}{4}(E_{11}+3E_{44}).
\end{eqnarray}
And $W^{\pm}_{1i1k}=0$, if $i\neq k$.
\end{lemma}

\begin{proof}
For $\delta W^\pm=0$, from \eqref{6-26-1}, we have
\begin{equation}
T^\pm_{ijk}=|\nabla V|W^\pm_{1ijk}.
\end{equation}
Let $i=k=2$ and $j=1$, then $j'=3$, $k'=4$. By \eqref{6-26-2},
\begin{equation}
\aligned
T^\pm_{212}=&|\nabla V|W^\pm_{1212}\\
=&\frac{3}{4}(-|\nabla V|E_{22})-\frac{1}{4}(|\nabla V|E_{11}).
\endaligned
\end{equation}
And by a similar argument we get \eqref{7-9-1} and \eqref{7-9-2}.
If $j=1$ and $i\neq k$, then $j'\neq1$, $k'\neq1$, so
\begin{equation*}
T^\pm_{i1k}=|\nabla V|W^\pm_{1i1k}=0.
\end{equation*}

\end{proof}

We will use the orthonormal basis above on the rest part of the paper.
For convenient, we denote $\lambda_1=W^{\pm}_{1212}$, $\lambda_2=W^{\pm}_{1313}$, $\lambda_3=W^{\pm}_{1414}$.
From Lemma \ref{lem6-19-1}, we get
\begin{equation}\label{6-24-7}
\aligned
|W^\pm|^2=&16(\lambda^2_1+\lambda^2_2+\lambda^2_3)\\
=&(E_{11}+3E_{22})^2+(E_{11}+3E_{33})^2+(E_{11}+3E_{44})^2\\
=&9|E|^2-12E^2_{11}.
\endaligned
\end{equation}

\section{Some lemmas}

\begin{lemma}\label{lem4}
Let $(M^4,g,V)$ be a  static triple with scalar
curvature 12 and $\delta W^\pm=0$. If $\partial M$ is isometric to a standard sphere $\mathbb{S}^3$,
then $W^\pm=0$ on $\partial M$.
\end{lemma}

\begin{proof}
We assume that $\partial M$ has constant sectional curvature $K$,
since it is isometric to a standard sphere $\mathbb{S}^3$. For
$\partial M=\{p \in M: V(p)=0\}$, then $\partial M$ is  totally
geodesic. So from Gauss equation and Codazzi equation, we know
\begin{eqnarray}
\label{11}&&R_{ijkl}=K(\delta_{ik}\delta_{jl}-\delta_{il}\delta_{jk}),\\
\label{21}&&R_{1ijk}=0
\end{eqnarray}
for $i,j,k,l>1$. Then,
\begin{eqnarray*}
12=R&=&R_{11}+R_{22}+R_{33}+R_{44}\\&=&R_{11}+R_{1212}+2K+R_{1313}+2K+R_{1414}+2K
\\&=&2R_{11}+6K,
\end{eqnarray*}
which implies
\begin{equation}\label{6-25-2}
E_{11}=3-3K.
\end{equation}
For
\begin{equation*}
R_{22}=R_{1212}+R_{3232}+R_{4242},
\end{equation*}
we have
\begin{equation*}
R_{1212}=E_{22}+3-2K,
\end{equation*}
which implies
\begin{equation*}
\aligned
W_{1212}&=R_{1212}-\frac{1}{2}(E_{11}+E_{22})-1\\&=\frac{1}{2}(E_{22}-E_{11})+2-2K.
\endaligned
\end{equation*}
Then
\begin{equation}\label{6-25-2-2}
W^\pm_{1212}=\frac{1}{2}W_{1212}=\frac{1}{4}(E_{22}-E_{11})+1-K.
\end{equation}
Combining \eqref{6-25-1}  with \eqref{6-25-2-2}, we get
$$E_{22}=K-1.$$
Then, using \eqref{6-25-2}, we obtain $W^\pm_{1212}=0$. With a same
argument, we know $W^\pm_{1313}=W^\pm_{1414}=0$. Then we finish the
proof.
\end{proof}

\begin{lemma}
Let $(M^4,g,V)$ be a  static triple with scalar
curvature 12 and $\delta W^\pm=0$, then the following equations
hold:
\begin{equation}\label{6-24-1}
|W^{\pm}|^2V_l+12V_jE_{ik}W^{\pm}_{jilk}=0,
\end{equation}
\begin{equation}\label{6-29-1}
E_{ik,j}W^{\pm}_{ijkl}=0,
\end{equation}
and
\begin{equation}\label{6-19-5}
\langle\nabla|W^{\pm}|^2,\nabla
V\rangle-4V|W^{\pm}|^2+12VE_{jl}E_{ik}W^{\pm}_{ijkl}=0.
\end{equation}
\end{lemma}

\begin{proof}
If $\nabla V=0$, then \eqref{6-24-1} holds true. If $\nabla V\neq0$,
to prove \eqref{6-24-1}, we just need to check
\begin{equation*}
|W^{\pm}|^2V_1+12V_1E_{ik}W^{\pm}_{1i1k}=0
\end{equation*}
and
\begin{equation*}
V_1E_{ik}W^{\pm}_{1ilk}=0 \quad \mbox{for} \quad  2\leq l\leq 4,
\end{equation*}
which are equivalent to
\begin{equation*}
E_{ik}W^{\pm}_{1i1k}=-\frac{1}{12}|W^{\pm}|^2
\end{equation*}
and
\begin{equation*}
E_{ik}W^{\pm}_{1ilk}=0 \quad \mbox{for} \quad  2\leq l\leq 4.
\end{equation*}
We know from Lemma \ref{lem6-19-1}
\begin{eqnarray*}
E_{ik}W^{\pm}_{1i1k}&=&E_{ii}W^{\pm}_{1i1i}\\
&=&-\frac{1}{3}(4W^{\pm}_{1i1i}+E_{11})W^{\pm}_{1i1i}\\
&=&-\frac{1}{12}|W^{\pm}|^2
\end{eqnarray*}
and
\begin{eqnarray*}
& &E_{ik}W^{\pm}_{1i2k}\\
&=&E_{33}W^{\pm}_{1323}+E_{44}W^{\pm}_{1424}\\
&=&\pm E_{33}W^{\pm}_{1314}\pm E_{44}W^{\pm}_{1431}=0.
\end{eqnarray*}
Similarly, we have $E_{ik}W^{\pm}_{1i3k}=E_{ik}W^{\pm}_{1i4k}=0$. So
we finish the proof of \eqref{6-24-1}. Furthermore, since
$V_{ijk}W^{\pm}_{ijkl}=0$, we get from the equation \eqref{1}
\begin{equation*}
\aligned
V_sR_{sikj}W^{\pm}_{ijkl}&=(V_{ikj}-V_{ijk})W^{\pm}_{ijkl}=V_{ikj}W^{\pm}_{ijkl}\\&=
VE_{ik,j}W^{\pm}_{ijkl}+V_jE_{ik}W^{\pm}_{ijkl},
\endaligned
\end{equation*}
which implies in view of the definition of $W$,
\begin{equation}\label{9-111}
\aligned
VE_{ik,j}W^{\pm}_{ijkl}=V_sW_{sikj}W^{\pm}_{ijkl}-\frac{3}{2}V_jE_{ik}W^{\pm}_{ijkl}.
\endaligned
\end{equation}
By a direct calculation,
\begin{equation*}
\aligned
W_{sikj}W^{\pm}_{ijkl}=&-W_{skji}W^{\pm}_{ijkl}-W_{sjik}W^{\pm}_{ijkl}\\
=&-W_{ijks}W^{\pm}_{ijkl}-W_{sikj}W^{\pm}_{ijkl}.
\endaligned
\end{equation*}
It is known that a symmetric linear transformation on the space of 2-forms
$\Lambda^{2}$ commutes with the Hodge star $*$
if and only if its Ricci contraction is proportional to $g$ (see \cite{K2} or Theorem 1.3 in \cite{ST}).
Since $W^\pm$ commutes with the Hodge star $*$, hence, so does $W^\pm\circ W^\pm$.  Thus we have
\begin{equation*}
W_{ijks}W^{\pm}_{ijkl}=W^\pm_{ijks}W^{\pm}_{ijkl}=\frac{|W^\pm|^2}{4}g_{ls},
\end{equation*}
then
\begin{equation}
W_{sikj}W^{\pm}_{ijkl}=-\frac{|W^\pm|^2}{8}g_{ls}.
\end{equation}
Taking the equation above into \eqref{9-111} arrives
\begin{equation*}
\aligned
VE_{ik,j}W^{\pm}_{ijkl}=-\frac{|W^\pm|^2}{8}V_l-\frac{3}{2}V_jE_{ik}W^{\pm}_{ijkl}=0.
\endaligned
\end{equation*}
So \eqref{6-29-1} holds when $V\neq0$. Since for static triple $V$ vanish on $\partial M$, so from Lemma \ref{lem4} we know $W^\pm=0$.
Then \eqref{6-29-1} holds when $V=0$.
At last, taking divergence on both sides of \eqref{6-24-1}, we have
\eqref{6-19-5}.
\end{proof}

\begin{lemma}\label{lem6-19-2}
Let $(M^4,g,V)$ be a connected complete static triple with scalar
curvature 12 and $\delta W^\pm=0$. If $\nabla V\neq0$ at $p\in M$,
then we have at $p$
\begin{equation}\label{6-19-2}
E_{jl}E_{ik}W^{\pm}_{ijkl}=\frac{4}{3}\det
W^{\pm}-\frac{2}{9}E_{11}|W^{\pm}|^2.
\end{equation}
\end{lemma}
\begin{proof}
If $\nabla V\neq0$, then we have
\begin{eqnarray*}
E_{jl}E_{ik}W^{\pm}_{ijkl}&=&E_{ii}E_{jj}W^{\pm}_{ijij}\\
&=&2E_{11}\sum_{i\neq1}E_{ii}W^{\pm}_{1i1i}+2\sum_{1<i<j}E_{ii}E_{jj}W^{\pm}_{ijij}.
\end{eqnarray*}
From Lemma \ref{lem6-19-1} , we have
\begin{eqnarray*}
& &\sum_{i<j}E_{ii}E_{jj}W^{\pm}_{ijij}\\
&=&E_{22}E_{33}W^{\pm}_{1414}+E_{22}E_{44}W^{\pm}_{1313}+E_{33}E_{44}W^{\pm}_{1212}\\
&=&\frac{1}{9}[(4\lambda_1+E_{11})(4\lambda_2+E_{11})\lambda_3+(4\lambda_1+E_{11})(4\lambda_3+E_{11})\lambda_2\\
& &+(4\lambda_2+E_{11})(4\lambda_3+E_{11})\lambda_1]\\
&=&\frac{2}{3}\det W^{\pm}-\frac{1}{36}E_{11}|W^{\pm}|^2
\end{eqnarray*}
in view of $\lambda_1+\lambda_2+\lambda_3=0$. Thus \eqref{6-19-2}
holds.
\end{proof}

\begin{lemma}\label{lem6-21-1}
Let $(M^4,g,V)$ be a connected complete static triple with scalar
curvature 12 and $\delta W^\pm=0$. If $W^{\pm}=0$ on $\partial M$,
then
\begin{equation}\label{6-19-3}
\aligned
&\int_MV(7|W^{\pm}|^2E(\nabla V,\nabla V)-24|\nabla V|^2\det W^{\pm})dv\\
=&3\int_MV|W^{\pm}|^2|\nabla V|^2dv.
\endaligned
\end{equation}
\end{lemma}

\begin{proof}
Since $W^{\pm}=0$ on $\partial M$, then we have by the divergence
theorem together with \eqref{6-24-1} and \eqref{6-29-1}
\begin{eqnarray*}
0&=&12\int_M(|\nabla V|^2V_jE_{ik}W^{\pm}_{ijkl})_{,l}dv\\
&=&12\int_M|\nabla
V|^2V_{jl}E_{ik}W^{\pm}_{ijkl}dv-2\int_M|W^{\pm}|^2\nabla^2V(\nabla
V,\nabla V)dv.
\end{eqnarray*}
Thus, we have by the equation \eqref{1}
\begin{eqnarray*}
12\int_M|\nabla
V|^2(E_{jl}-g_{jl})E_{ik}W^{\pm}_{ijkl}dv=2\int_M|W^{\pm}|^2(E(\nabla
V,\nabla V)-|\nabla V|^2)dv.
\end{eqnarray*}
Then, noticing that $E_{ii}W^{\pm}_{ijij}=0$, we get from Lemma
\ref{lem6-19-2}
\begin{eqnarray*}
& &\int_MV(16|\nabla V|^2\det W^{\pm}-\frac{8}{3}E(\nabla V,\nabla V)|W^{\pm}|^2)dv\\
&=&2\int_MV|W^{\pm}|^2(E(\nabla V,\nabla V)-|\nabla V|^2)dv.
\end{eqnarray*}
Thus \eqref{6-19-3} holds.
\end{proof}

\begin{lemma}\label{lem6-21-2}
Let $(M^4,g,V)$ be a connected complete static triple with scalar
curvature 12 and $\delta W^\pm=0$. If $W^{\pm}=0$ on $\partial M$,
then
\begin{equation}\label{6-19-8}
\aligned
&2\int_MV|W^{\pm}|^2E(\nabla V,\nabla V)dv\\
=&-\int_MV^3|E|^2|W^{\pm}|^2dv-\int_MV^2E(\nabla|W^{\pm}|^2,\nabla
V)dv.
\endaligned
\end{equation}
\end{lemma}

\begin{proof}
Since $W^{\pm}=0$ on $\partial M$, it can be easily get by the
divergence theorem
\begin{equation}\label{6-24-2}
\aligned
&  \int_MV|W^{\pm}|^2E_{ij}V_iV_jdv\\=&-\int_M(V|W^{\pm}|^2E_{ij}V_i)_{, j}Vdv\\
=&-\int_MV|W^{\pm}|^2E_{ij}V_iV_jdv-\int_MV^3|E|^2|W^{\pm}|^2dv\\
&-\int_MV^2E(\nabla|W^{\pm}|^2,\nabla V)dv
\endaligned
\end{equation}
in view of $E_{ij,j}=0$.
\end{proof}

\begin{lemma}\label{lem6-24-1}
Let $(M^4,g,V)$ be a connected complete static triple with scalar
curvature 12 and $\delta W^\pm=0$, then
\begin{eqnarray*}
& &\int_MV|\nabla V|^2|W^{\pm}|^2dv\\
&=&\int_MV^3(|\nabla W^{\pm}|^2+8|W^{\pm}|^2-72\det
W^{\pm})dv+\int_MV^2\langle\nabla V,\nabla|W^\pm|^2\rangle dv.
\end{eqnarray*}
\end{lemma}
\begin{proof}
Using the equation \eqref{2}, we can obtain
\begin{equation*}
\aligned
&\int_MV|\nabla V|^2|W^{\pm}|^2dv\\=&\int_MV(\frac{1}{2}\Delta(V^2)-V\Delta V)W^{\pm}|^2dv\\
=&\int_MV\bigg(\frac{1}{2}\triangle (V^2)+4V^2\bigg)|W^{\pm}|^2dv\\
=&\frac{1}{2}\int_MV^2\triangle(V|W^{\pm}|^2)dv+4\int_MV^3|W^{\pm}|^2dv,
\endaligned
\end{equation*}
where we use the Green formulas to get the last equation by $V=0$ on
$\partial M$. So, we get the conclusion from \eqref{w}.

\end{proof}

\section{Proof of Theorem \ref{mainthm}}

We denote $F=\frac{1}{2}(V^2+|\nabla V|^2)$, then differentiating it and we get by the equation \eqref{1}
\begin{equation*}
F_i=VE_{ij}V_j.
\end{equation*}
Since $E_{ij,i}=-R_ig_{ij}=0$,
differentiating it again and we arrive
\begin{equation}\label{lapF}
\triangle F=V^2|E|^2+E(\nabla V,\nabla V).
\end{equation}
in view of the equations \eqref{1} and \eqref{2}.
If $W^\pm=0$ on $\partial M$, then we get by Green formulas
\begin{eqnarray}\label{6-19-10}
\int_MV\triangle(F|W^\pm|^2)dv+4\int_MVF|W^\pm|^2dv=0,
\end{eqnarray}
which implies
\begin{eqnarray*}
0&=&\int_MV\bigg(V^2|E|^2+E(\nabla V,\nabla V)\bigg)|W^\pm|^2 dv+2\int_MV^2E(\nabla|W^{\pm}|^2,\nabla V)dv\\
& &+\int_M V(V^2+|\nabla V|^2)(|\nabla W^{\pm}|^2+8|W^\pm|^2-72det W^\pm)dv.
\end{eqnarray*}
in view of the equation \eqref{w} and \eqref{lapF}. Thus,
\begin{equation}\label{6-24-3}
\aligned
0\geq&\int_MV\bigg(E(\nabla V,\nabla V)|W^\pm|^2+8|\nabla V|^2|W^\pm|^2\\
&-72|\nabla V|^2\det W^\pm\bigg)dv+2\int_M V^2E(\nabla|W^{\pm}|^2,\nabla V)dv\\
&+\int_MV^3(|E|^2|W^\pm|^2+8|W^\pm|^2-72\det W^\pm)dv.
\endaligned
\end{equation}
From Lemma \ref{lem6-21-1}, we obtain
\begin{equation*}
\aligned
&\int_MV\bigg(E(\nabla V,\nabla V)|W^\pm|^2+8|\nabla V|^2|W^\pm|^2
-72|\nabla V|^2\det W^\pm\bigg)dv\\=&\int_MV\bigg(17|\nabla V|^2|W^\pm|^2
-20E(\nabla V,\nabla V)|W^\pm|^2\bigg)dv
\endaligned
\end{equation*}
From Lemma \ref{lem6-21-2}, we obtain
\begin{equation*}
\aligned
&-20\int_MVE(\nabla V,\nabla V)|W^\pm|^2dv\\
=&10\int_MV^3
|E|^2|W^\pm|^2dv+10\int_MV^2
E(\nabla V,\nabla |W^\pm|^2)dv.
\endaligned
\end{equation*}
Taking the two equations above into \eqref{6-24-3}, we arrive
\begin{equation}\label{9-21-1}
\aligned
0\geq&\int_M\bigg(17V|\nabla V|^2|W^\pm|^2+12V^2E(\nabla|W^{\pm}|^2,\nabla V)\\
&+V^3(11|E|^2|W^\pm|^2+8|W^\pm|^2-72\det W^\pm)\bigg)dv.
\endaligned
\end{equation}
Using Lemma \ref{lem6-24-1}, \eqref{9-21-1} becomes
\begin{equation}\label{6-24-6}
\aligned
0\geq&\int_M\bigg(16V|\nabla V|^2|W^\pm|^2+12V^2E(\nabla|W^{\pm}|^2,\nabla V)\\
&+V^3(11|E|^2|W^\pm|^2+16|W^\pm|^2-144\det W^\pm)\\
&+V^2\langle\nabla V,\nabla|W^\pm|^2\rangle\bigg)dv.
\endaligned
\end{equation}
If $\nabla V\neq0$, we have from \eqref{6-19-5} and Lemma \ref{lem6-19-2},
\begin{equation}\label{6-24-5}
\langle\nabla|W^{\pm}|^2,\nabla V\rangle=4V|W^{\pm}|^2-16V\det W^{\pm}+\frac{8}{3}VE_{11}|W^{\pm}|^2,
\end{equation}
which implies
\begin{equation}\label{6-19-6}
\aligned
&E(\nabla|W^{\pm}|^2,\nabla V)\\
=&4VE_{11}|W^{\pm}|^2-16VE_{11}\det W^{\pm}+\frac{8}{3}VE^2_{11}|W^{\pm}|^2.
\endaligned
\end{equation}
We denote by
$$M_1=\{p\in M, \nabla V(p)\neq0\}.$$
By Sard theorem, we know
$$Vol(M \backslash M_1)=0.$$
So, we can get by putting \eqref{6-24-5} and \eqref{6-19-6} into \eqref{6-24-6}
\begin{equation}
\aligned
0\geq&\int_MV^3\bigg(16\frac{1}{V^2}|\nabla V|^2|W^\pm|^2\\&+48E_{11}|W^{\pm}|^2
-192E_{11}\det W^\pm+32E^{2}_{11}|W^\pm|^2\\
&+(11|E|^2|W^\pm|^2+16|W^\pm|^2-144\det W^\pm)\\
&+4|W^{\pm}|^2-16\det W^{\pm}+\frac{8}{3}E_{11}|W^{\pm}|^2\bigg)dv.
\endaligned
\end{equation}
By the inequality (3.12) in \cite{Gur}, we know
\begin{equation}
|\det W^\pm|\leq\frac{1}{24\sqrt{6}}|W^\pm|^3.
\end{equation}
Combining with \eqref{6-24-7}, we get
\begin{eqnarray*}
0&\geq&16\int_MV|\nabla V|^2|W^\pm|^2dv\\&&+\int_{M}V^3|W^\pm|^2\bigg(\frac{11}{9}
|W^\pm|^2+(48+\frac{8}{3})E_{11}+(32+\frac{132}{9})E^2_{11}\\
& &-\frac{8}{\sqrt{6}}|E_{11}||W^\pm|-\frac{20}{3\sqrt{6}}|W^\pm|+20\bigg)dv\\
&\geq&\int_{M}V^3|W^\pm|^2\bigg(\frac{5}{9}|W^\pm|^2-\frac{8}{\sqrt{6}}|E_{11}||W^\pm|+\frac{24}{5}E^2_{11})\\
& &+(\frac{2}{3}|W^\pm|^2-\frac{20}{3\sqrt{6}}|W^\pm|+\frac{25}{9})\\
& &+(\frac{628}{15}E^2_{11}+\frac{152}{3}E^2_{11}+\frac{155}{9})\bigg)dv\geq0.
\end{eqnarray*}
So $W^\pm=0$ if $\nabla V\neq 0$, then we get $W^\pm=0$ on $M$
from the continuity of $W^\pm$. By \eqref{7-9-6}, we know $T^\pm=T=0$.
Then with a similar discussion of Lemma 4.2 in \cite{CC} (or Lemma 4 in \cite{BDR}), we know $W=0$, i.e. $M$ is locally conformally flat.
Then Theorem \ref{mainthm} follows from  Theorem \ref{thmK} (Theorem 3.1 in\cite{K}).
In fact, Theorem \ref{thm2} follows from  Theorem \ref{thmK} as well.

\end{document}